\newif\ifModelAomart
\IfFileExists{aomart.cls}{%
  \ModelAomarttrue
  \documentclass[twoside]{aomart}%
}{%
  \ModelAomartfalse
  \documentclass[11pt,a4paper,twoside]{article}%
}
\usepackage[T1]{fontenc}
\usepackage[utf8]{inputenc}
\usepackage{lmodern,amsmath,amssymb,amsthm,mathrsfs,booktabs,needspace,array,enumitem}
\usepackage{xcolor,microtype,geometry,indentfirst}
\usepackage{tikz}
\usetikzlibrary{arrows.meta,calc,positioning}
\usepackage{fancyhdr}
\usepackage{hyperref}
\usepackage[capitalise,nameinlink,noabbrev]{cleveref}

\definecolor{ContentsBlue}{RGB}{20,65,165}
\definecolor{CitationGreen}{RGB}{0,115,45}
\definecolor{ReferenceRed}{RGB}{180,25,35}
\hypersetup{colorlinks=true,linkcolor=ReferenceRed,citecolor=CitationGreen,
  urlcolor=ContentsBlue,linktoc=all,
  pdftitle={A uniform quantitative lower bound for Pythagorean pairs in boxes of prime powers}}

\ifModelAomart
  \copyrightnote{}
  \doinumber{}
\else
  \providecommand{\copyrightnote}[1]{}
  \providecommand{\doinumber}[1]{}
  \newcommand{\ModelAddress}{}
  \newcommand{\ModelEmail}{}
  \newcommand{\ModelKeywords}{}
  \newcommand{\ModelSubjClass}{}
  \newcommand{\ModelThanks}{}
  \newcommand{\address}[1]{\gdef\ModelAddress{#1}}
  \newcommand{\email}[1]{\gdef\ModelEmail{#1}}
  \newcommand{\keywords}[1]{\gdef\ModelKeywords{#1}}
  \newcommand{\subjclass}[2][]{\gdef\ModelSubjClass{#2}}
  \renewcommand{\thanks}[1]{%
    \gdef\ModelThanks{#1}%
  }
\fi

\fancypagestyle{plain}{%
  \fancyhf{}%
  \fancyfoot[LE,RO]{\footnotesize\thepage}%
  }
\fancypagestyle{firstpage}{%
  \fancyhf{}%
  \fancyfoot[LE,RO]{\footnotesize\thepage}%
  }

\newtheoremstyle{paperplain}{.65\baselineskip}{.65\baselineskip}
  {\itshape}{0pt}{\normalfont\bfseries}{.}{.5em}{}
\newtheoremstyle{paperdefinition}{.65\baselineskip}{.65\baselineskip}
  {\normalfont}{0pt}{\normalfont\bfseries}{.}{.5em}{}
\theoremstyle{paperplain}
\newtheorem{theorem}{Theorem}[section]

\newtheorem{lemma}[theorem]{Lemma}

\theoremstyle{paperdefinition}

\newtheorem{remark}[theorem]{Remark}

\makeatletter
\@namedef{subjclassname@2020}{2020 Mathematics Subject Classification}
\renewcommand\section{\@startsection{section}{1}{\z@}
  {1.2\baselineskip plus .3\baselineskip minus .2\baselineskip}
  {.6\baselineskip}{\normalfont\large\bfseries\raggedright}}
\renewcommand\subsection{\@startsection{subsection}{2}{\z@}
  {.9\baselineskip plus .2\baselineskip minus .1\baselineskip}
  {-.5em}{\normalfont\bfseries}}
\renewcommand\subsubsection{\@startsection{subsubsection}{3}{\z@}
  {.7\baselineskip}{-.5em}
  {\normalfont\itshape}}
\renewenvironment{proof}[1][\proofname]{\par\pushQED{\qed}%
  \normalfont\topsep6\p@\@plus3\p@\relax\trivlist
  \item[\hskip\labelsep\itshape #1\@addpunct{.}]\ignorespaces
}{\popQED\endtrivlist\@endpefalse}
\makeatother

\crefname{theorem}{Theorem}{Theorems}
\crefname{lemma}{Lemma}{Lemmas}
\crefname{proposition}{Proposition}{Propositions}
\crefname{corollary}{Corollary}{Corollaries}
\crefname{section}{Section}{Sections}
\crefname{subsection}{Section}{Sections}
\crefname{equation}{Equation}{Equations}
\Crefname{theorem}{Theorem}{Theorems}
\Crefname{lemma}{Lemma}{Lemmas}
\Crefname{proposition}{Proposition}{Propositions}
\Crefname{corollary}{Corollary}{Corollaries}
\Crefname{section}{Section}{Sections}
\Crefname{subsection}{Section}{Sections}
\Crefname{equation}{Equation}{Equations}

\newcommand{\Fq}{\mathbb F_q}
\newcommand{\N}{\mathbb N}
\newcommand{\Z}{\mathbb Z}

\newcommand{\R}{\mathbb R}
\newcommand{\C}{\mathbb C}

\newcommand{\leg}[2]{\left(\frac{#1}{#2}\right)}
\newcommand{\ee}{\mathrm e}
\newcommand{\cond}{\operatorname{cond}}
\newcommand{\1}{\mathbf 1}
\newcommand{\Zzero}{\mathcal Z}

\newcommand{\LauWuCitation}{\cite[Lemmas~4.1 and~4.3; Equation~(5.7)]{LauWu}}

\numberwithin{equation}{section}
\allowdisplaybreaks[1]
\title{Quantitative lower bound for Pythagorean pairs in boxes of prime powers}
\newcommand*{\AuthorName}{Rongzhong Xiao}
\newcommand*{\AuthorAffiliation}{School of Mathematical Sciences, University of Science and Technology of China, Hefei, Anhui, 230026, PR China}
\newcommand*{\AuthorEmail}{xiaorz@mail.ustc.edu.cn}
\newcommand{\FundingText}{This research is supported by the National Natural
Science Foundation of China (No. 12426201), and by the Postdoctoral
Fellowship Program of CPSF under Grant Number GZB20260746.}
\ifx\AuthorName\empty\else\author{\AuthorName}\fi
\ifx\AuthorAffiliation\empty\else\address{\AuthorAffiliation}\fi
\ifx\AuthorEmail\empty\else\email{\AuthorEmail}\fi
\date{}
\keywords{Pythagorean pairs, Quadratic characters, Quadratic nonresidues}
\subjclass[2020]{Primary 11B30; Secondary 11N37, 05D10}
\thanks{\FundingText}

\newcommand{\PaperAbstract}{%
For each positive integer \(N\), let \(p_i\) be the \(i\)-th prime for \(i\in\{1,\ldots,N\}\), and set
\[
 \Phi_N=\left\{\prod_{i=1}^{N}p_i^{a_i}:a_i\in\{0,1,\ldots,N-1\}\right\}.
\]
Write \(f(N)\) for the largest cardinality of a subset of \(\Phi_N\)
containing no two distinct legs of an integral right triangle.  We prove
\[
 N^{N-1}(\log\log\log N)\ll f(N).
\]
}

\newcommand{\FallbackMetadata}{%
  \ifModelAomart\else
    \begingroup
    \renewcommand{\thefootnote}{}
    \makebox[0pt]{\footnotemark}
    \endgroup
    \footnotetext{%
      \textit{2020 Mathematics Subject Classification:} \ModelSubjClass.\par
      \textit{Keywords:} \ModelKeywords.\par
      \ModelThanks
    }%
  \fi
}

\newcommand{\FinalContact}{%
  \ifModelAomart\else
    \bigskip
    \begingroup
    \small
    \noindent\textit{Address:} \ModelAddress\par
    \noindent\textit{E-mail:} \href{mailto:\ModelEmail}{\ModelEmail}\par
    \endgroup
  \fi
}

\begin{document}
\ifModelAomart
  \begin{abstract}\PaperAbstract\end{abstract}
  \maketitle
\else
  \maketitle
  \begin{abstract}\PaperAbstract\end{abstract}
  \FallbackMetadata
\fi
\thispagestyle{plain}

\section{Introduction}
\label{sec:introduction}

\subsection{History and motivations}

The equation $x^2+y^2=z^2$ occupies a prominent place in the study of partition regularity. Erd\H{o}s and Graham \cite[Section~3, Problem~6]{GrahamFavorite} asked whether every finite colouring of $\N$ must contain a monochromatic solution of this equation.

The two-colour problem was initially investigated by computation.  Cooper and Overstreet \cite{CooperOverstreet} produced, by a SAT computation, a colouring of $\{1,\ldots,7664\}$ with no monochromatic Pythagorean triple.  Subsequently, Heule, Kullmann, and Marek \cite[Theorem~1]{HeuleKullmannMarek} established by a certified cube-and-conquer computation, accompanied by a formally checkable certificate of approximately $200$ terabytes, that $\{1,\ldots,7824\}$ still admits such a colouring, whereas every two-colouring of $\{1,\ldots,7825\}$ contains a monochromatic Pythagorean triple.  Two recent independent preprints provide entirely theoretical proofs of the same two-colour result.  Leng \cite{Leng2026} uses analytic and ergodic methods, while Wesley \cite{Wesley2026} uses an elementary argument combining linear algebra with counting.  Thus the two-colour question for finite initial intervals has been resolved; the analogous problem for three or more colours is still open.

A different direction is based on multiplicative harmonic analysis.  Frantzikinakis and Host \cite[Corollary~2.8 and Problem~3]{FrantzikinakisHost} developed higher-order Fourier methods for bounded multiplicative functions. They proved that the following holds: Suppose that $a$, $b$, $c$, and $a+b$ are squares of positive integers. For the equation $ax^2+by^2=cz^2$,
every
finite colouring of $\N$ admits a solution $x,y,z\in\N$ with
$x\ne y$ and $x,y$ of the same colour.
Their work also singled out the two-leg Pythagorean problem: find $x$ and $y$ of the same colour such that $x^2+y^2=z^2$ for some $z$.  Frantzikinakis, Klurman, and Moreira \cite[Theorems~1.1--1.2]{FKM} later proved that every finite colouring of $\N$ contains two distinct same-coloured legs with this property, and they established the corresponding assertions for the other two pairs of entries in a Pythagorean triple.  Their density theorem gives the same conclusion in every set of positive multiplicative density. 

In subsequent work, Frantzikinakis, Klurman, and Moreira \cite[Theorems~1.2--1.4 and 1.9--1.10]{FKMgeneralized} extended the same framework to generalized Pythagorean equations
$ax^2+by^2=cz^2$.  They obtained unconditional partition and density regularity for the relevant pairs whenever $ac$ or $bc$ is a square.  In the complementary case in which $(a+b)c$ is a square, they proved the corresponding assertions assuming an Elliott-type correlation conjecture, and they also obtained an unconditional result for pretentious multiplicative level sets.  These theorems place the Pythagorean-pair problem in a wider theory of homogeneous quadratic forms.

The multiplicative viewpoint has continued to develop.  Azevedo and Moreira \cite{AzevedoMoreira} proved that finitely many completely multiplicative functions can simultaneously be close to $1$ on the three entries of a Pythagorean triple, and they treated square coefficients satisfying Rado's condition.  Frantzikinakis and Mountakis \cite{FrantzikinakisMountakis} established an analogous recurrence theorem for generalized triples $ax^2+by^2=cz^2$ in pretentious systems.  Leng, McNamara, and Mountakis \cite{LengMcNamaraMountakis} obtained almost-Pythagorean quadruples $x^2+y^2w^2=z^2$ in multiplicatively dense sets, allowing $w$ to belong to a second dense set.

The purpose of the present paper is to make this picture quantitative in a finite model.  We replace an arbitrary multiplicative F\o lner sequence by the explicit prime-power boxes introduced below and prove a decay estimate for the largest subset containing no two distinct legs of an integral right triangle.

\subsection{Prime boxes and the main result}

For a positive integer $N$, let $p_1<p_2<\cdots<p_N$ denote the first $N$
primes.  We work in the prime-power box
\begin{equation*}
 \Phi_N=\left\{\prod_{i=1}^{N}p_i^{a_i}:a_i\in\{0,1,\ldots,N-1\}\right\}.
\end{equation*}
Unique factorisation gives $|\Phi_N|=N^N$, and the map
\[
 (a_1,\ldots,a_N)\longmapsto\prod_{i=1}^Np_i^{a_i}
\]
identifies this set with the box of exponent vectors $\{0,\ldots,N-1\}^N$.  It is easy to see that $\Phi=(\Phi_N)_{N\in\N}$ is a multiplicative F\o lner sequence on $\N$. 

Call a set $A\subseteq\Phi_N$ admissible when no two distinct
elements of $A$ are the legs of an integral right triangle; that is, there
are no distinct $u,v\in A$ and no $w\in\N$ satisfying $u^2+v^2=w^2$.
Define $f:\N\to\N$ by
\begin{equation*}
 f(N)=\max\{|A|:A\subseteq\Phi_N\text{ is admissible}\}.
\end{equation*}
The hypotenuse $w$ is unrestricted; in particular, it need not lie in $A$
or in the same prime-power box.  The elementary congruence obstruction
already gives $f(N)\geq N^{N-1}$: restrict to the elements whose exponent
of $2$ is zero.  They are all odd, and the sum of two odd squares is
$2\pmod4$, so it cannot be a square.  The following theorem gives a
logarithmic improvement:

\begin{theorem}\label{thm:main}
There is an absolute positive integer $N_0$ such that, for every $N\in\N$ with $N\geq N_0$,
\begin{equation}\label{eq:main}
 f(N)\geq \frac{a\kappa}{24}\,N^{N-1}(\log \log \log N),
\end{equation}
where the positive real constants $a,b,\kappa$ are
\[
 a=\frac12\log 2,\qquad b=4\log 2,\qquad
 \kappa=\frac{1}{32(1+c_2b)}.
\]
Here $c_2>0$ denotes the absolute constant in part~\textup{(b)} of
\Cref{thm:lau-wu}.
\end{theorem}
Combining the above lower bound and \cite[Theorem~1.1]{XiaoPrimeBoxes}, we have
$$N^{N-1}(\log \log \log N)\ll f(N)\ll N^N(\log\log N)^{-1/1296}.$$ 

\medskip
\noindent\textbf{Structure of the paper.}
\Cref{sec:proof-outline} outlines the proof of the main result.
\Cref{sec:preliminaries} records the prime estimates and quadratic characters
used throughout.  \Cref{sec:analytic} states the required bounds for zeros
of Dirichlet $L$-functions and proves a smoothed explicit formula.
\Cref{sec:omission} removes a possible exceptional character by omitting
one odd prime from the allowed conductors.  \Cref{sec:finite-field}
constructs a set of residues in which every sum of two squares is a
quadratic nonresidue.  \Cref{sec:main-proof} combines these ingredients
and proves \Cref{thm:main}.

\medskip
\noindent\textbf{Notation.}
We write $\N=\{1,2,\ldots\}$, $\Z_{\geq0}=\{0,1,2,\ldots\}$,
and use $\Z,\R,\C$ for the integers, real numbers, and complex numbers.
The notation $\R_{>0}$ denotes the positive real numbers, and
$\mathcal P=\{2,3,5,\ldots\}$ is the set of all prime numbers.
For $p\in\mathcal P$ and $n\in\N$, let
\begin{equation*}
 v_p(n)=\max\{k\in\Z_{\geq0}:p^k\mid n\}.
\end{equation*}

All logarithms are natural.  For $t\in\R_{>0}$ and $j\in\N$, set
\[
 \log_1t=\log t,\qquad \log_{j+1}t=\log(\log_jt),
\]
where the second expression is used only when $\log_jt>0$ and all
preceding iterates are defined.  

The notations
$A\lesssim_{d_1,\ldots,d_k}B$ and $A\ll_{d_1,\ldots,d_k}B$ mean
$A\leq C_{d_1,\ldots,d_k}B$, where
$C_{d_1,\ldots,d_k}\in\R_{>0}$ depends only on those data.
The notation
$A\asymp_{d_1,\ldots,d_k}B$ means both comparison inequalities.
For a real or complex quantity $A$ and $B\geq0$ real,
$A=O_{d_1,\ldots,d_k}(B)$ means $|A|\ll_{d_1,\ldots,d_k}B$.
The notation $A=o(B)$ means $A/B\to0$ in the stated limit, with $B>0$.
For positive quantities $A$ and $B$, we write $A\sim B$ when
$A/B\to1$ in the stated limit.

\medskip
\noindent\textbf{Declaration on the use of AI.}
The mathematical results, arguments, and proofs in this paper 
were developed by the author, who takes full responsibility for the final manuscript.
AI tools were used only for language polishing and to assist with literature searches . 

\medskip
\noindent\textbf{Conflict of Interest Statement.} The author states that there is no conflict of interest.

\section{Proof outline of the main result}
\label{sec:proof-outline}

The proof combines a congruence obstruction with estimates for
quadratic characters. The congruence obstruction produces an
admissible set once a suitable prime is available. The analytic
argument finds that prime while allowing one exception among the
prescribed quadratic residues. The finite-field construction is
designed to tolerate this exception.

\medskip\noindent
\textbf{The congruence obstruction and the desired scale.}
Suppose that $q>p_N$ is prime and that, for a large integer $K$,
a set $S\subset\Fq^\times$ satisfies
\begin{equation}\label{eq:outline-finite-field}
 |S|\geq K/6-1,\qquad
 \leg{s^2+t^2}{q}=-1\quad(s,t\in S).
\end{equation}
Multiplication by $c\in\Fq^\times$ preserves the second property,
because $c^2$ is a nonzero square. Thus
$B_c=\{n\in\Phi_N:n\bmod q\in cS\}$ is admissible: an equality
$u^2+v^2=w^2$ with $u,v\in B_c$ would make a quadratic nonresidue
modulo $q$ equal to a square. Every element of $\Phi_N$ is nonzero
modulo $q$ and belongs to exactly $|S|$ of these sets. Consequently,
\begin{equation}\label{eq:outline-average}
 f(N)\geq\frac{N^N|S|}{q-1}.
\end{equation}
This averaging argument requires no information about how $\Phi_N$
is distributed among residue classes. Let $y$ be the small-prime
cutoff to be chosen below. If $p_N<q<2p_N$ and $K\asymp y$,
Chebyshev's estimates give a lower bound of order
$N^{N-1}y/\log p_N$. We therefore seek a prime of this size for
which $y$ can exceed $\log p_N$ by a growing factor.

\medskip\noindent
\textbf{Constructing the residues.}
The required prime will satisfy $q\equiv5\pmod8$, with every odd
prime at most $y$ a quadratic residue modulo $q$, except possibly
one prime. Put $K=\lfloor y\rfloor<q$. Start with the integers
$u\in[1,K]$ satisfying $u\equiv1\pmod4$. For any two such integers,
$(u+v)/2$ is an odd integer at most $K$. If neither $u$ nor
$(u+v)/2$ is divisible by the possible omitted prime, both are
quadratic residues modulo $q$. Since $\displaystyle \leg{2}{q}=-1$, the identity
$u+v=2((u+v)/2)$ makes $u+v$ a quadratic nonresidue.

If the omitted prime is $r$, first remove multiples of $r$.
Then partition the nonzero residue classes modulo $r$ into pairs
$\{a,-a\}$. In each pair, choose the class represented by more
of the remaining integers; if the counts agree, choose either class.
This retains at least half of the integers and ensures
$r\nmid u+v$ for all selected $u,v$. The resulting set $U$ satisfies
$|U|\geq K/12-1/2$. Taking both square roots modulo $q$ of each
$u\in U$ gives \eqref{eq:outline-finite-field}. Thus omitting one
prime costs only a fixed proportion of the construction.

\medskip\noindent
\textbf{Removing the possible exceptional character.}
For sufficiently large $x$, choose
\begin{equation}\label{eq:outline-scale}
 L=\log x,\qquad J=\log L,\qquad T=\log J,\qquad
 y=\kappa LT=\kappa\log x\,\log_3x.
\end{equation}
Let $\displaystyle P_y=\prod_{p\in\mathcal P,\,p\leq y}p$, $H=\log P_y$, and let
$\mathcal C_y$ be the family of primitive real characters whose
conductors divide $4P_y$. The Lau--Wu theorem (see \Cref{thm:lau-wu}) permits at most one
zero in its specified region near $1$, and that zero must be real.
Such a zero could contribute a term comparable with $x$ in the
explicit formula, so we exclude its character before summing the errors.

\Cref{lem:deletion} shows that this character, if it exists,
has an odd prime divisor in its conductor. Omitting this prime
from $P_y$ excludes the character. Denote the omitted set by $E_y$,
with $|E_y|\leq1$. Retain precisely those characters in $\mathcal C_y$
whose conductors are not divisible by any prime in $E_y$, and denote
this family by $\mathcal D_y$.
Its $L$-functions have no zeros with $\Re s\geq1-\eta(y)$ and
$|\Im s|\leq H$, where $\eta(y)=c_1\sqrt{\log H}/(2H)$ as in
\Cref{thm:lau-wu}. The preceding construction explains why this
analytic omission is compatible with the congruence argument.

\medskip\noindent
\textbf{Controlling the sum of the character errors.}
We need an estimate for all characters together:
\begin{equation}\label{eq:outline-aggregate}
 \sum_{\chi\in\mathcal D_y}
 \left|T_x(\chi)-\1_{\chi=\chi_1}xW(1)\right|=o(x),
\end{equation}
Here $\chi_1$ is the principal character, and $T_x$, $W$, and
the weight $h$ are defined in \Cref{lem:explicit}.
An error of size $o(x)$ for each character separately would not
suffice, because the number of characters grows with $x$.

The choice \eqref{eq:outline-scale} makes the required sum small.
Indeed, $H\asymp LT$, $\log H\sim J$, and
$\log|\mathcal C_y|=O(LT/J)$.
With $\varepsilon=J^{-1/2}$, this last quantity is
$o(\varepsilon L)$. Hence the factor $x^{-\varepsilon}$ absorbs
the number of characters when bounding zeros with
$\Re s<1-\varepsilon$ and $|\Im s|\leq H$.
For zeros nearer $1$ in the same height range, the Lau--Wu bounds
control their total number across the family. Combining these
bounds with the absence of zeros for $\Re s\geq1-\eta(y)$ bounds
their contribution by a constant times
\begin{equation*}
 x\sqrt{\log H}
 \left(\ee^{-3\varepsilon L/4}+\ee^{-3\eta(y)L/4}\right)=o(x).
\end{equation*}
Here $\eta(y)L\asymp\sqrt J/T$ grows faster than
$\log\log H\sim T$, as does $\varepsilon L=L/\sqrt J$.
The exponential decay of $W$ controls zeros with $|\Im s|>H$,
and the explicit-formula errors also sum to $o(x)$.
These estimates establish \eqref{eq:outline-aggregate}.

\medskip\noindent
\textbf{Finding the prime and completing the argument.}
The product $F_y$ in \eqref{eq:F-definition} encodes the required
residue conditions through factors of the form $1+\chi(n)$ or
$1-\chi(n)$. Every factor is nonnegative. Expanding the product
gives exactly the characters of $\mathcal D_y$, each with coefficient
$1$ or $-1$, and the principal character has coefficient $1$.
Thus \eqref{eq:outline-aggregate} yields
\begin{equation*}
 \sum_{p\in\mathcal P}(\log p)h(p/x)F_y(p)
 =xW(1)+o(x)>0.
\end{equation*}
The weight satisfies $h(t)\leq0$ outside $(1,2)$. Positivity
therefore supplies a prime $q\in(x,2x)$ with $F_y(q)>0$.
The factors force $q\equiv5\pmod8$ and
$\displaystyle \leg{r}{q}=1$ for every odd prime $r\leq y$ outside $E_y$.

Finally, take $x=p_N$, apply the finite-field construction, and
use \eqref{eq:outline-average}. The scale
$y=\kappa\log p_N\,\log_3p_N$ gives the factor $\log_3N$.
\Cref{sec:main-proof} supplies the estimates and obtains
the constant $a\kappa/24$.

\section{Preliminaries}
\label{sec:preliminaries}

\subsection{Prime estimates}
For $t\in\R_{>0}$, let
\[
 \pi(t)=\#\{p\in\mathcal P:p\leq t\},\qquad
 \vartheta(t)=\sum_{\substack{p\in\mathcal P\\p\leq t}}\log p,\qquad 
 \psi(t)=\sum_{\substack{p\in\mathcal P,\ j\in\N\\p^j\leq t}}\log p.
\]
We shall use only the following elementary form of Chebyshev's estimates.

\begin{lemma}\label{lem:cheb}
There is a real constant \(t_0\geq2\) such that, for every real \(t\geq t_0\),
\begin{equation}\label{eq:cheb}
 at\leq\vartheta(t)\leq bt,
 \qquad
 \frac{at}{\log t}\leq\pi(t)\leq\frac{8t}{\log t},
\end{equation}
where \(a=\frac12\log2\) and \(b=4\log2\).
\end{lemma}

\begin{proof}
We first compare $\psi$ with central binomial coefficients, then remove
the higher prime powers and estimate $\pi$.

\medskip
\noindent\textbf{Step 1. The binomial comparison.}
For $p\in\mathcal P$ and $m\in\N$, let $v_p(m)$ denote the exponent
of $p$ in the prime factorisation of $m$.  Counting multiples of each
power of $p$ in $m!$ gives
\begin{equation*}
 v_p(m!)=
 \sum_{\substack{j\in\N\\p^j\leq m}}
 \left\lfloor\frac{m}{p^j}\right\rfloor.
\end{equation*}
Indeed, a factor of $m!$ divisible by $p^r$ but not by $p^{r+1}$,
where $r\in\Z_{\geq0}$, is counted once for each
$j\in\{1,\ldots,r\}$.  Empty sums are understood to be zero.
Unique factorisation therefore gives, for every $n\in\N$,
\begin{equation*}
 \log\binom{2n}{n}
 =\sum_{\substack{p\in\mathcal P\\p\leq2n}}
   \sum_{j=1}^{\lfloor\log(2n)/\log p\rfloor}
 \left(
 \left\lfloor\frac{2n}{p^j}\right\rfloor
 -2\left\lfloor\frac{n}{p^j}\right\rfloor
 \right)\log p.
\end{equation*}
Each parenthesised expression is $0$ or $1$, since
$\lfloor2u\rfloor-2\lfloor u\rfloor\in\{0,1\}$ for every real
$u\geq0$.  It equals $1$ whenever $n<p^j\leq2n$.  Thus
\begin{equation}\label{eq:cheb-binomial-psi}
 \psi(2n)-\psi(n)
 \leq\log\binom{2n}{n}\leq\psi(2n)
 \qquad(n\in\N).
\end{equation}
The central coefficient is the largest of the coefficients
$\displaystyle \binom{2n}{k}$, $k\in\{0,\ldots,2n\}$: their successive ratio is
$(2n-k)/(k+1)$ for $k\in\{0,\ldots,2n-1\}$, which is greater than
$1$ for $k<n$ and less than $1$ for $k\geq n$.  The binomial theorem gives
\begin{equation}\label{eq:cheb-binomial-size}
 \sum_{k=0}^{2n}\binom{2n}{k}=2^{2n},
 \qquad
 \frac{2^{2n}}{2n+1}\leq\binom{2n}{n}\leq2^{2n}.
\end{equation}

\medskip
\noindent\textbf{Step 2. Linear bounds for $\psi$.}
For real $t\geq2$, set $n=\lfloor t/2\rfloor\in\N$.
Since $t-2\leq2n\leq t$, the monotonicity of $\psi$ and
\eqref{eq:cheb-binomial-psi}--\eqref{eq:cheb-binomial-size} imply
\begin{equation}\label{eq:cheb-psi-lower}
 \psi(t)\geq\psi(2n)
 \geq2n\log2-\log(2n+1)
 \geq(t-2)\log2-\log(t+1).
\end{equation}
For the upper bound, let $u\geq1$ be real and put
$J=\lfloor\log u/\log2\rfloor\in\Z_{\geq0}$.
Then $2^J\leq u<2^{J+1}$.  Apply
\eqref{eq:cheb-binomial-psi}--\eqref{eq:cheb-binomial-size}
with $n=2^j$ for each $j\in\{0,\ldots,J\}$.
Using $\psi(1)=0$, we obtain
\begin{align}
 \psi(u)
 \leq\psi(2^{J+1})
 =\sum_{j=0}^{J}\bigl(\psi(2^{j+1})-\psi(2^j)\bigr)
\leq(\log2)\sum_{j=0}^{J}2^{j+1}
 =(2^{J+2}-2)\log2
 \leq 4u\log2.
 \label{eq:cheb-psi-upper}
\end{align}
In particular, $\vartheta(u)\leq\psi(u)\leq bu$ for every real $u\geq1$.

\medskip
\noindent\textbf{Step 3. Removing higher prime powers.}
For real $t\geq2$, let
$K=\lfloor\log t/\log2\rfloor\in\N$.
No prime power $p^j\leq t$ can have $j>K$, since $p\geq2$.
Grouping the terms of $\psi(t)$ by their exponent, and applying
\eqref{eq:cheb-psi-upper}, gives
\begin{align*}
 0\leq\psi(t)-\vartheta(t)
 =\sum_{j=2}^{K}\vartheta(t^{1/j})
 \leq b\sum_{j=2}^{K}t^{1/j}
 \leq b(K-1)\sqrt t
 \leq4\sqrt t\log t.
\end{align*}
Here $t^{1/j}\leq\sqrt t$ for every integer $j\geq2$, and
$K\leq\log t/\log2$; the sums are empty when $K=1$.
Together with \eqref{eq:cheb-psi-lower}, this yields
\begin{equation}\label{eq:cheb-theta-lower}
 \vartheta(t)\geq t\log2-E(t),\qquad
 E(t)=2\log2+\log(t+1)+4\sqrt t\log t.
\end{equation}
Since $E(t)/t\to0$ as $t\to\infty$, there is a real $t_1\geq2$
such that $E(t)\leq at$ for all real $t\geq t_1$.
As $\log2=2a$, \eqref{eq:cheb-psi-upper} and
\eqref{eq:cheb-theta-lower} give
\begin{equation*}
 at\leq\vartheta(t)\leq bt\qquad(t\geq t_1).
\end{equation*}

\medskip
\noindent\textbf{Step 4. Bounds for the prime-counting function.}
For real $t\geq t_1$, each prime $p\leq t$ satisfies
$\log p\leq\log t$, so
\begin{equation*}
 \pi(t)\geq\frac{\vartheta(t)}{\log t}
 \geq\frac{at}{\log t}.
\end{equation*}
For the upper bound, there are at most $\sqrt t$ primes at most
$\sqrt t$, and every prime $p>\sqrt t$ satisfies
$\log p>\tfrac12\log t$.  Consequently,
\begin{align*}
 \pi(t)
 =\pi(\sqrt t)+
   \sum_{\substack{p\in\mathcal P\\\sqrt t<p\leq t}}1
 \leq\sqrt t+
   \frac{2}{\log t}
   \sum_{\substack{p\in\mathcal P\\\sqrt t<p\leq t}}\log p
 \leq\left(8\log2+\frac{\log t}{\sqrt t}\right)
       \frac{t}{\log t}.
\end{align*}
Because $\log2<1$ and $(\log t)/\sqrt t\to0$, we may choose a real
$t_0\geq t_1$ such that $$(\log t)/\sqrt t\leq8(1-\log2)$$ for
every real $t\geq t_0$.  The last display then gives
$\pi(t)\leq8t/\log t$, completing the proof.
\end{proof}

\subsection{Quadratic characters}
\label{sec:quadratic-characters}
For real $y\geq100$, put
\begin{equation}\label{eq:PyH}
 P_y=\prod_{\substack{p\in\mathcal P\\p\leq y}}p,
 \qquad H_y=\log P_y=\vartheta(y).
\end{equation}
When $y$ is fixed, we also write $H=H_y$.
We use the standard discriminant notation
\begin{equation*}
 \chi_D(n)=\leg{D}{n}\qquad(n\in\Z),
 \qquad \cond(\chi_D)=|D|,
\end{equation*}
where the symbol is the Kronecker symbol and $D\in\Z\setminus\{0\}$
is a fundamental discriminant.  Explicitly, either $D$ is squarefree
and $D\equiv1\pmod4$, or $D=4b$ with $b\in\Z$ squarefree and
$b\equiv2,3\pmod4$.  We include $D=1$, for which $\chi_1(n)=1$
for every $n\in\Z$.  The correspondence $D\mapsto\chi_D$ is a
bijection onto the primitive real Dirichlet characters; see
\cite[Section~1]{LauWu}.

The family needed below is
\begin{equation}\label{eq:Cy}
 \mathcal C_y=
 \{\chi_D:D\text{ is a fundamental discriminant},\ |D|\mid4P_y\}.
\end{equation}
Its members can be listed explicitly.  For each positive divisor
$d$ of $P_y/2$, define
\begin{equation*}
 d^*=(-1)^{(d-1)/2}d
     =\prod_{\substack{r\in\mathcal P, r>2\\r\mid d}}r^*,
 \qquad r^*=(-1)^{(r-1)/2}r\quad(r\in\mathcal P,\ r>2).
\end{equation*}
Here $d$ is odd and squarefree, $d^*\equiv1\pmod4$, and $1^*=1$.
The product identity follows by multiplying the signs for the odd
prime factors of $d$.  The four possibilities are
\begin{center}
\begin{tabular}{@{}ccc@{}}
\toprule
Discriminant $D$ & Character & Conductor $|D|$\\
\midrule
$d^*$ & $\chi_{d^*}$ & $d$\\
$-4d^*$ & $\chi_{-4d^*}$ & $4d$\\
$8d^*$ & $\chi_{8d^*}$ & $8d$\\
$-8d^*$ & $\chi_{-8d^*}$ & $8d$\\
\bottomrule
\end{tabular}
\end{center}
Indeed, the definition of a fundamental discriminant permits only
the exponents $0,2,3$ at $2$.  For each odd divisor $d$ of $P_y/2$, the congruence
condition fixes the sign when the exponent is $0$ or $2$, and permits
both signs when it is $3$.  This proves that the table is exhaustive.
It has no repetitions: writing $|D|=2^j d$ with $d$ odd
uniquely determines $d$, and the exponent $j$ and the sign of $D$
distinguish the four rows.  Consequently,
\begin{equation}\label{eq:character-family-size}
 R_y=|\mathcal C_y|=4\cdot2^{\pi(y)-1}=2^{\pi(y)+1},
 \qquad \cond(\chi_D)\leq8d\leq4P_y.
\end{equation}
In particular, $\chi_1$ occurs exactly once.  For example, $d=3$
gives the discriminants $-3,12,-24,24$.  The last two characters
have the same conductor $24$, but opposite parity, since
$\chi_D(-1)=\operatorname{sgn}D$.

Multiplicativity of the Kronecker symbol gives the identities
\begin{equation}\label{eq:quadratic-products}
\begin{aligned}
 \chi_{d^*}&=\prod_{\substack{r\in\mathcal P, r>2\\r\mid d}}\chi_{r^*},
 &\chi_{-4d^*}&=\chi_{-4}\chi_{d^*},\\
 \chi_{8d^*}&=\chi_8\chi_{d^*},
 &\chi_{-8d^*}&=\chi_{-4}\chi_8\chi_{d^*}.
\end{aligned}
\end{equation}
These are identities on all of $\Z$.  In the last identity,
$\chi_{-4}\chi_8=\chi_{-8}$: on odd integers this follows from
$(4/n)=1$, and on even integers both sides vanish.

\section{Analytic input}
\label{sec:analytic}

For a Dirichlet character $\chi$ modulo $Q$, its Dirichlet
$L$-function is initially defined, for $\Re s>1$, by the absolutely
convergent series and Euler product
\begin{equation*}
 L(s,\chi)=\sum_{n=1}^{\infty}\frac{\chi(n)}{n^s}
 =\prod_{p\in\mathcal P}\left(1-\frac{\chi(p)}{p^s}\right)^{-1}.
\end{equation*}
It has a meromorphic continuation to $\C$: it is entire when $\chi$ is
nonprincipal, while for a principal character it has a single simple pole
at $s=1$.  In particular, for the principal primitive character
$\chi_1$ of conductor $1$, $L(s,\chi_1)=\zeta(s)$.  All zeros below are
counted with their orders of vanishing.  These standard properties are
recorded, for example, in \cite[Sections~11.5, 12.1, 12.6, and~12.10]{Apostol}.

We need two analytic inputs. The first bounds the zeros of the
$L$-functions associated with $\mathcal C_y$. The second expresses a
weighted sum over primes in terms of these zeros.

\subsection{The Lau--Wu estimates}
\begin{theorem}[Lau--Wu, \LauWuCitation]\label{thm:lau-wu}
There are absolute real constants \(c_1,c_2,C>0\) and \(y_0\geq100\) such that the following assertions hold for every real \(y\geq y_0\).  Enlarge \(y_0\), if necessary, so that \(H_y>\ee^4\).

\begin{enumerate}[label=\textup{(\alph*)}]
\item Put
\[
 \eta(y)=\frac{c_1\sqrt{\log H_y}}{2H_y}.
\]
The product $\displaystyle \prod_{\chi\in\mathcal C_y}L(s,\chi)$, for $s\in\C$, has at most one zero, counted with multiplicity, in
\[
 \Re s\geq1-2\eta(y),\qquad |\Im s|\leq H_y.
\]
If such a zero exists, it is real.

\item Define
\[
 \mathcal L_1(s;y)=
 \prod_{\substack{d\in\N\\d\mid P_y/2}}
 L(s,\chi_{d^*})L(s,\chi_{8d^*}),
 \qquad
 \mathcal L_4(s;y)=
 \prod_{\substack{d\in\N\\d\mid P_y/2}}
 L(s,\chi_{-4d^*})L(s,\chi_{-8d^*}).
\]
For \(\ell\in\{1,4\}\) and \(0\leq\alpha\leq1\), let
\(N_\ell(\alpha;y)\) count, with
multiplicity, the zeros of \(\mathcal L_\ell(s;y)\) in
\(\alpha\leq\Re s\leq1\) and \(|\Im s|\leq H_y\).
Put
\[
 k=\lfloor\sqrt{\log H_y}\rfloor,\quad k\geq2,
 \quad \eta_1(y)=\frac{k}{2(2^k-2)}.
\]
For real \(0<t<1/2\),
\[
 N_\ell(1-t;y)\leq
 \begin{cases}
 C\exp\!\left(\dfrac{c_2tH_y}{\sqrt{\log H_y}}+\dfrac12\log_2 H_y\right),
       &t\leq\eta_1(y),\\[1.2ex]
 C\exp\!\left(\dfrac{c_2tH_y}{\log(1/t)}\right),
       &t>\eta_1(y).
 \end{cases}
\]
\item If \(\chi\in\mathcal C_y\) has conductor \(Q\in\N\) and \(T\in\R\) satisfies \(T\geq2\), then
\[
 \#\{\rho\in\C:L(\rho,\chi)=0,\ 0<\Re\rho<1,\ |\Im\rho|\leq T\}
 \ll T\log(Q(T+2)),
\]
with multiplicity.  A possible zero at \(s=0\) contributes at most one.
\end{enumerate}
\end{theorem}

\subsection{A smoothed explicit formula}
The next lemma relates weighted prime sums to zeros of $L(s,\chi)$.

\begin{lemma}\label{lem:explicit}
Define $h:(0,\infty)\to\R$ and
$W:\{s\in\C:\Re s>-35\}\to\C$ by
\[
 h(t)=t^{35}\ee^{-24t}(t-1)(2-t),
 \qquad
 W(s)=\int_0^\infty h(t)t^{s-1}\,dt.
\]
For a primitive Dirichlet character \(\chi:\Z\to\C\) of conductor \(Q\in\N\), define
\[
 T_x(\chi)=\sum_{p\in\mathcal P}(\log p)\chi(p)h(p/x),
 \qquad x\in\R,\quad x\geq2.
\]
Let $\Zzero_\chi$ be the multiset consisting of all nontrivial zeros
$\rho\in\C$ of $L(s,\chi)$ with $0<\Re\rho<1$, together with
the zero $\rho=0$ if $L(0,\chi)=0$.  Each zero is repeated according
to its order of vanishing. Then
\begin{equation}\label{eq:explicit}
 T_x(\chi)=\1_{\chi=\chi_1}\,xW(1)
 -\sum_{\rho\in\Zzero_\chi}W(\rho)x^\rho
 +O\!\left(\sqrt{x}\log^2(x+2)+x^{-1/2}\log(Q+2)\right).
\end{equation}
The sum over zeros is absolutely convergent.  Moreover, there is a real constant $C_h>0$, depending only on $h$, such that
\begin{equation}\label{eq:Wdecay}
 W(1)>0,
 \qquad
 |W(\sigma+i\tau)|\leq C_h\ee^{-|\tau|}
 \quad\left(\sigma,\tau\in\R,\ -\frac12\leq\sigma\leq2\right).
\end{equation}
\end{lemma}
\begin{proof}
We first evaluate the Mellin transform, then apply the residue theorem
to the weighted sum over prime powers, and finally remove the higher
prime powers.  All implied constants below are absolute, since $h$
is fixed.

\medskip\noindent
\textbf{Step 1. The Mellin transform.}
The function $h$ is positive on $(1,2)$ and nonpositive on
$(0,1]\cup[2,\infty)$.  It satisfies $h(t)=O(t^{35})$ as $t\downarrow0$
and decays exponentially as $t\to\infty$.  Hence the integral defining
$W(s)$, and its derivatives with respect to $s$, converge locally
uniformly for $\Re s>-35$.  Thus $W$ is holomorphic there.  Expanding
$(t-1)(2-t)=-t^2+3t-2$ and integrating each term gives
\begin{equation*}
\begin{aligned}
 W(s)&=-24^{-s-37}\Gamma(s+37)
       +3\,24^{-s-36}\Gamma(s+36)
       -2\,24^{-s-35}\Gamma(s+35)\\
 &=24^{-s-37}\Gamma(s+35)(108+s-s^2).
\end{aligned}
\end{equation*}
Here $\Gamma$ is Euler's gamma function.  In particular,
\begin{equation*}
 W(1)=108\,24^{-38}\Gamma(36)
     >0.
\end{equation*}
The vertical-strip form of Stirling's formula
\cite[pp.~87--88 and~293--295]{Olver}, applied to $s+35$, gives
\begin{equation}\label{eq:W-stirling}
 |W(\sigma+i\tau)|
 \ll(1+|\tau|)^{77/2}\ee^{-\pi|\tau|/2}
 \ll\ee^{-|\tau|}
 \qquad\left(-\tfrac12\leq\sigma\leq2,\ \tau\in\R\right).
\end{equation}
The estimate is uniform because $\Re(s+35)\in[34.5,37]$;
bounded $\tau$ are covered by continuity.  This proves
\eqref{eq:Wdecay}.

\medskip\noindent
\textbf{Step 2. A Mellin integral for the prime powers.}
For $n\in\N$, let $\Lambda(n)=\log p$ if $n=p^k$ for a prime
$p\in\mathcal P$ and an integer $k\geq1$, and let $\Lambda(n)=0$
otherwise.  Put
\begin{equation}\label{eq:weighted-von-mangoldt}
 S_x(\chi)=\sum_{n=1}^{\infty}\Lambda(n)\chi(n)h(n/x).
\end{equation}
This series is absolutely convergent.  Write $\displaystyle \int_{(c)}$ for upward
integration along $s=c+i\tau$, $\tau\in\R$, with $c\in\R$.
Mellin inversion \cite[Chapter~4, Appendix; Equation~(4.106)]{IwaniecKowalski}
gives
\begin{equation*}
 h(t)=\frac1{2\pi i}\int_{(2)}W(s)t^{-s}\,ds
 \qquad(t\in\R,\ t>0).
\end{equation*}
Its hypotheses hold because $h$ is continuous, of bounded variation,
and of Mellin type $(-35,\infty)$: the displayed asymptotics give
$h(t)=O(t^{35})$ at $0$ and $h(t)=O(t^{37}\ee^{-24t})$ at infinity,
while $\displaystyle\int_0^\infty|h'(t)|\,dt<\infty$.  The inversion integral is
absolutely convergent by \eqref{eq:W-stirling}.  Substitution into
\eqref{eq:weighted-von-mangoldt} is justified by the finite majorant
\begin{equation*}
 x^2\sum_{n=1}^{\infty}\frac{\Lambda(n)}{n^2}
       \int_{-\infty}^{\infty}|W(2+i\tau)|\,d\tau<\infty.
\end{equation*}
Since $$-L'/L(s,\chi)=\sum_{n=1}^{\infty}\Lambda(n)\chi(n)n^{-s}$$
for $\Re s>1$, it follows that
\begin{equation}\label{eq:mellin}
 S_x(\chi)=\frac1{2\pi i}\int_{(2)}
       -\frac{L'}{L}(s,\chi)W(s)x^s\,ds.
\end{equation}
Here $L'$ denotes differentiation with respect to $s\in\C$.

\medskip\noindent
\textbf{Step 3. Bounds for the contour shift.}
Let $a_\chi\in\{0,1\}$ satisfy $\chi(-1)=(-1)^{a_\chi}$, and let
$\overline\chi(n)=\overline{\chi(n)}$ for $n\in\Z$.
The functional equation for a primitive character of conductor $Q$ is
\begin{equation*}
\begin{aligned}
 \Lambda_\chi(s)&=\left(\frac Q\pi\right)^{(s+a_\chi)/2}
       \Gamma\!\left(\frac{s+a_\chi}{2}\right)L(s,\chi),\\
 \Lambda_\chi(s)&=\varepsilon_\chi\Lambda_{\overline\chi}(1-s),
       \qquad\varepsilon_\chi\in\C,\quad|\varepsilon_\chi|=1;
\end{aligned}
\end{equation*}
see \cite[Theorem~4.15; Equations~(4.73)--(4.75)]{IwaniecKowalski}.
For $\chi=\chi_1$ and $Q=1$, this is the meromorphic functional
equation of $\zeta(s)$.  Thus it also applies in the principal case.
Set $\psi_\Gamma(z)=\Gamma'(z)/\Gamma(z)$ for
$z\in\C\setminus\{0,-1,-2,\ldots\}$.  Logarithmic differentiation
gives the meromorphic identity
\begin{equation*}
\begin{aligned}
 \frac{L'}{L}(s,\chi)
 ={}-\frac{L'}{L}(1-s,\overline\chi)-\log\frac Q\pi
 -\frac12\psi_\Gamma\!\left(\frac{s+a_\chi}{2}\right)
  -\frac12\psi_\Gamma\!\left(\frac{1-s+a_\chi}{2}\right).
\end{aligned}
\end{equation*}
On $s=-1/2+i\tau$, the reflected logarithmic derivative has an
absolutely convergent Dirichlet series on $\Re(1-s)=3/2$ and is
bounded by $$\sum_{n=1}^\infty\Lambda(n)n^{-3/2}<\infty.$$
The digamma estimate \cite[pp.~87--88 and~293--295]{Olver} bounds the other
terms by $O(\log(|\tau|+2))$.  For bounded $\tau$, their arguments
stay a fixed distance from the poles of $\psi_\Gamma$.
Consequently,
\begin{equation*}
 \left|\frac{L'}{L}\left(-\frac12+i\tau,\chi\right)\right|
 \ll\log\bigl(Q(|\tau|+2)\bigr)\qquad(\tau\in\R),
\end{equation*}
and the new vertical integral satisfies
\begin{equation*}
 \left|\frac1{2\pi i}\int_{(-\frac{1}{2})}
       -\frac{L'}{L}(s,\chi)W(s)x^s\,ds\right|
 \ll x^{-\frac{1}{2}}\int_{-\infty}^{\infty}
       \ee^{-|\tau|}\log\bigl(Q(|\tau|+2)\bigr)\,d\tau
 \ll x^{-\frac{1}{2}}\log(Q+2).
\end{equation*}

We also justify the disappearance of the horizontal sides.
Let $\mathcal Z_\chi^\circ$ denote the multiset of nontrivial zeros
$\rho$ with $0<\Re\rho<1$.  The standard local zero estimates are
\begin{align}
 \#\{\rho\in\mathcal Z_\chi^\circ:|\Im\rho-\tau|\leq1\}
 &\ll\log\bigl(Q(|\tau|+2)\bigr),
 \label{eq:local-zero-count}\\
 \frac{L'}{L}(s,\chi)
 &=\sum_{\substack{\rho\in\mathcal Z_\chi^\circ\\|s-\rho|<1}}
       \frac1{s-\rho}
       +O\!\left(\log\bigl(Q(|\tau|+2)\bigr)\right),
 \label{eq:local-log-derivative}
\end{align}
where $s=\sigma+i\tau$, $-1/2\leq\sigma\leq2$, $|\tau|\geq2$,
and $s$ is not a zero in \eqref{eq:local-log-derivative}.
Zeros are counted with multiplicity.  These are
\cite[Proposition~5.7; Equations~(5.27)--(5.28)]{IwaniecKowalski}:
primitive Dirichlet $L$-functions have degree one and analytic
conductor comparable to $Q(|\tau|+3)$, as verified in
\cite[Section~5.9]{IwaniecKowalski}.  Thus the constants are uniform
in $Q$ and $\chi$.

For each integer $j\geq3$, let $M_j$ count the zeros in
$\mathcal Z_\chi^\circ$ with $j-1\leq|\Im\rho|\leq j+2$.
Covering these two intervals by finitely many intervals of length
$2$ and using \eqref{eq:local-zero-count} gives
$M_j\ll\log(Q(j+2))$.  Set $\delta_j=1/(4(M_j+1))$.
Intervals of radius $\delta_j$ about the $M_j$ absolute ordinates
remove length at most $2M_j\delta_j<1/2$ from $[j,j+1]$.
Choose $T_j$ outside them.  Then both $T_j$ and $-T_j$ have distance
at least $\delta_j$ from every zero ordinate; ordinates outside the
specified intervals have distance at least $1$. 
\eqref{eq:local-zero-count}--\eqref{eq:local-log-derivative} yield
\begin{equation*}
 \sup_{-1/2\leq\sigma\leq2}
 \left|\frac{L'}{L}(\sigma\pm iT_j,\chi)\right|
 \ll\log^2(Q(j+2)).
\end{equation*}
Each horizontal integral in the rectangle with vertical sides
$\Re s=2,-1/2$ and heights $\pm T_j$ is therefore
\begin{equation*}
 O\!\left(x^2\ee^{-j}\log^2(Q(j+2))\right)=o(1)
 \qquad(j\to\infty),
\end{equation*}
with $x$ and $Q$ fixed.

\medskip\noindent
\textbf{Step 4. The residues and the zero sum.}
The only pole of $L(s,\chi)$ is the simple pole at $s=1$ when
$\chi=\chi_1$.  Dirichlet $L$-functions have no zeros on
$\Re s=1$; for nonprincipal primitive characters this is
\cite[Theorem~5.26]{IwaniecKowalski}, while for $\chi_1$ it is the
classical zero-free line for $\zeta(s)$, recorded in
\cite[Equation~(4.17.2), p.~89]{Titchmarsh} (with its simple pole at $s=1$).
The functional equation then shows that their only
possible zero on $\Re s=0$ is at $s=0$.  For a nonprincipal even
primitive character this zero is simple: the gamma factor has a
simple pole at $0$, whereas $\Lambda_{\overline\chi}(1)\ne0$.
For an odd primitive character the gamma factor is finite and
nonzero at $0$, so $L(0,\chi)\ne0$.  In the principal case,
$\zeta(0)=-1/2$.  The remaining trivial zeros lie at negative
integers, outside the rectangle.  Thus the poles of the integrand
in \eqref{eq:mellin} inside it are exactly the principal pole at $1$
and the zeros in $\Zzero_\chi$ with $|\Im\rho|<T_j$.

At the principal pole the residue of $-L'/L$ is $+1$.
At a zero $\rho$ of order $m\in\N$, it is $-m$, since locally
$L(s,\chi)=(s-\rho)^m g(s)$ with $g$ holomorphic and $g(\rho)\ne0$.
Their contributions are therefore $xW(1)$ and
$-mW(\rho)x^\rho$, respectively, including the possible zero at $0$.
There are no poles contributed by $W$, which is holomorphic on
the whole rectangle.

To pass to the limit, write
$N_\chi(T)=\#\{\rho\in\mathcal Z_\chi^\circ:|\Im\rho|\leq T\}$,
with multiplicity, for real $T\geq2$.
The bound $N_\chi(T)\ll T\log(Q(T+2))$ is the Dirichlet case of
\cite[Theorem~5.24]{IwaniecKowalski}.  By \eqref{eq:Wdecay},
\begin{equation*}
 \sum_{\rho\in\mathcal Z_\chi^\circ}|W(\rho)x^\rho|
 \ll x\sum_{j=0}^{\infty}\ee^{-j}N_\chi(j+2)
 \ll x\sum_{j=0}^{\infty}\ee^{-j}(j+2)\log(Q(j+4))<\infty.
\end{equation*}
The possible term at $0$ is finite.  The residue theorem, followed
by $j\to\infty$, now gives
\begin{equation}\label{eq:vonmangoldt}
 S_x(\chi)=\1_{\chi=\chi_1}xW(1)
 -\sum_{\rho\in\Zzero_\chi}W(\rho)x^\rho
 +O\!\left(x^{-1/2}\log(Q+2)\right).
\end{equation}

\medskip\noindent
\textbf{Step 5. Removing the higher prime powers.}
For real $Y>0$, define
\begin{equation*}
 A(Y)=\sum_{\substack{p\in\mathcal P,\ k\in\N\\k\geq2,\ p^k\leq Y}}
             \log p.
\end{equation*}
This is zero for $Y<4$.  The estimate $\vartheta(u)\ll u$ for
all real $u>0$ follows from \Cref{lem:cheb}, after increasing
the constant on a bounded interval.  For $Y\geq4$ it gives
\begin{equation*}
 A(Y)=\sum_{k=2}^{\lfloor\log Y/\log2\rfloor}\vartheta(Y^{1/k})
 \ll Y^{1/2}+Y^{1/3}\log Y\ll Y^{1/2}.
\end{equation*}
Indeed, every term with $k\geq3$ is $O(Y^{1/3})$, and
$\log Y\leq(6/\ee)Y^{1/6}$ for $Y\geq1$, by maximizing
$(\log Y)/Y^{1/6}$.  Since $|\chi(n)|\leq1$, partial summation yields
\begin{equation}\label{eq:prime-power-removal}
\begin{aligned}
 & |S_x(\chi)-T_x(\chi)|
 \leq\int_0^\infty|h(t/x)|\,dA(t)
 \leq\frac1x\int_0^\infty A(t)|h'(t/x)|\,dt \\
 & \hspace{6cm}
 \ll\sqrt x\int_0^\infty u^{1/2}|h'(u)|\,du
 \ll\sqrt x.
\end{aligned}
\end{equation}
Here $|h|$ is absolutely continuous with
$|(|h|)'|\leq|h'|$ almost everywhere.  The boundary terms vanish
by $A(t)=0$ near $0$ and the exponential decay at infinity.
The last integral is finite because $h'(u)=O(u^{34})$ near $0$
and $h'(u)=O(u^{37}\ee^{-24u})$ as $u\to\infty$.
Combining \eqref{eq:vonmangoldt} and \eqref{eq:prime-power-removal}
proves \eqref{eq:explicit}.
\end{proof}

\section{Omission of a possible exceptional conductor prime}
\label{sec:omission}

\Cref{thm:lau-wu}(a) allows at most one exceptional real zero.
We exclude its character by omitting an odd prime divisor of its conductor.
The remaining characters then have no zeros in the region needed for the
main proof. 

\begin{lemma}\label{lem:deletion}
For all sufficiently large real $y$, there is a set
\[
 E_y\subseteq\{r\in\mathcal P:r\leq y,\ r>2\},
 \qquad |E_y|\leq1,
\]
such that, with
\[
 P_y^*={P_y}/{\prod_{r\in E_y}r},
\]
and
\[
 \mathcal D_y=
 \left\{\chi_D:
 D\text{ is a fundamental discriminant},
 |D|\mid4P_y^*
\right\},
\]
$P_y^*\in\N$, and every $L(s,\chi)$ with $\chi\in\mathcal D_y$ is
zero-free in
\[
 \Re s\geq1-\eta(y),\qquad |\Im s|\leq H.
\]
\end{lemma}

\begin{proof}
Put
\begin{equation*}
 \Omega_y=\{s\in\C:\Re s\geq1-2\eta(y),\ |\Im s|\leq H\}.
\end{equation*}
By \Cref{thm:lau-wu}(a), the $L$-functions for $\chi\in\mathcal C_y$
have at most one zero in $\Omega_y$, counted with multiplicity.
If this zero exists, it is real.  The four characters in
$\mathcal C_y$ whose discriminants have no odd prime factor are
\[
 \chi_1,\qquad\chi_{-4},\qquad\chi_8,\qquad\chi_{-8}.
\]
For the three nonprincipal characters, $L(1,\chi)\ne0$ by \cite[Theorem~6.20, p.~141]{Apostol}; for the
principal character, $L(s,\chi_1)=\zeta(s)$ has a simple pole at $s=1$.
Consequently, for some fixed $\delta>0$, none of these four $L$-functions
has a zero in the disk $|s-1|<\delta$ (for $\chi_1$, use that
$(s-1)\zeta(s)$ is holomorphic and nonzero at $s=1$).  Since
$\eta(y)\to0$, we may assume
\[
 2\eta(y)<\min(\delta,1/2).
\]

If no $L(s,\chi)$ with $\chi\in\mathcal C_y$ has a zero in
$\Omega_y$, set $E_y=\varnothing$.
Then $P_y^*=P_y$ and $\mathcal D_y=\mathcal C_y$.
The conclusion follows because
$\{s\in\C:\Re s\geq1-\eta(y),\ |\Im s|\leq H\}\subseteq\Omega_y$.

Suppose now that some $L(s,\chi)$ with $\chi\in\mathcal C_y$ has a zero in $\Omega_y$.  By
\Cref{thm:lau-wu}(a), it is a unique real zero $\beta$ of a unique
character $\chi_*\in\mathcal C_y$, counted with multiplicity.  We have
$\beta\geq1-2\eta(y)>0$; the Euler product gives zero-freeness for
$\Re s>1$, and the nonvanishing at $s=1$ rules out $\beta=1$.  Hence
\[
 0<\beta<1,\qquad 1-\beta\leq2\eta(y)<\delta.
\]
Thus $|\beta-1|<\delta$, and the disk exclusion above shows that
$\chi_*$ is not one of the four characters with no odd prime factor in
its discriminant.

Write $\chi_*=\chi_{D_*}$. By the preceding paragraph, $|D_*|$
has an odd prime divisor; choose one and denote it by $r_y$.
Since $|D_*|\mid4P_y$ and $r_y>2$, we have $r_y\mid P_y$, hence
$r_y\leq y$. Set $E_y=\{r_y\}$ and $P_y^*=P_y/r_y$.
The integer $P_y$ is squarefree, so $r_y\nmid4P_y^*$.
Every $\chi_D\in\mathcal D_y$ satisfies $|D|\mid4P_y^*$ and therefore
$r_y\nmid|D|$. In contrast, $r_y\mid|D_*|$, so
$\chi_*\notin\mathcal D_y$.
The zero $\beta$ belonged only to $L(s,\chi_*)$, and no other
$L$-function in $\mathcal C_y$ has a zero in $\Omega_y$.
Since $\mathcal D_y\subseteq\mathcal C_y\setminus\{\chi_*\}$, every
$L(s,\chi)$ with $\chi\in\mathcal D_y$ is zero-free in $\Omega_y$.
In particular, it has no zero with $\Re s\geq1-\eta(y)$ and
$|\Im s|\leq H$, as required.  Finally, $P_y^*\in\mathbb N$ because $r_y$ is a prime factor of
$P_y$.
\end{proof}

\section{The finite-field construction}
\label{sec:finite-field}

The next lemma constructs a set $S\subset\Fq^\times$ such that
$s^2+t^2$ is a quadratic nonresidue for every $s,t\in S$.
We obtain $S$ by taking square roots of selected integers
$u\equiv1\pmod4$. The choice of these integers also accounts for
the possible omitted prime in \Cref{lem:deletion}.

\begin{lemma}\label{lem:residue}
Let $q\equiv5\pmod8$ be prime, let $K\in\N$ satisfy $1\leq K<q$,
and let $E$ be empty or a singleton containing an odd prime $r\leq K$.
Suppose
\begin{equation}\label{eq:residue-hypothesis}
 \leg{\ell}{q}=1
 \qquad(\ell\leq K,\ \ell\text{ an odd prime},\ \ell\notin E).
\end{equation}
Then there is a set $S\subset\Fq^\times$ such that
\begin{equation}\label{eq:Ssize}
 |S|\geq\frac K6-1
\end{equation}
and
\begin{equation}\label{eq:Sproperty}
 \leg{s^2+t^2}{q}=-1\qquad(s,t\in S).
\end{equation}
Every multiplicative translate $cS$, $c\in\Fq^\times$, has the same
property.
\end{lemma}

\begin{proof}
We divide the proof into three steps.

\medskip\noindent
\textbf{Step 1. Selecting the integers.}
Put
\begin{equation}\label{eq:U-initial}
 U_0=\{u\in\N:1\leq u\leq K,\ u\equiv1\pmod4\},
 \qquad |U_0|=\left\lfloor\frac{K+3}{4}\right\rfloor\geq\frac K4.
\end{equation}
If $E=\varnothing$, take $U=U_0$.  Suppose $E=\{r\}$.  Since
$\gcd(4,r)=1$, the simultaneous conditions $u\equiv1\pmod4$ and
$u\equiv0\pmod r$ specify one residue class modulo $4r$.  Thus the
number of such integers in $[1,K]$ is at most $K/(4r)+1$.  Set
$V=\{u\in U_0:r\nmid u\}$; then
\begin{equation*}
 |V|\geq\frac K4-\frac K{4r}-1.
\end{equation*}

The nonzero residue classes modulo $r$ split into disjoint pairs
$\{a,-a\}$.  Indeed, $a=-a$ would give $2a=0$, hence $a=0$,
because $r$ is odd.  For each pair \(\{a,-a\}\) of nonzero residue classes modulo \(r\), choose the class represented by more elements of \(V\). If the two classes are represented equally often, choose either one.
Let $A\subset(\Z/r\Z)^\times$ be the retained classes and put
$U=\{u\in V:u\bmod r\in A\}$.  The construction gives
\begin{equation}\label{eq:U-no-opposites}
 A\cap(-A)=\varnothing,
 \qquad r\nmid u,\quad r\nmid(u+v)\qquad(u,v\in U).
\end{equation}
At least half of $V$ is retained.  Since
$r\geq3$, we obtain
\begin{equation*}
 |U|\geq\frac{|V|}{2}
 \geq\frac K8\left(1-\frac1r\right)-\frac12
 \geq\frac K{12}-\frac12.
\end{equation*}
This last bound also holds when $E=\varnothing$, by
\eqref{eq:U-initial}.

\medskip\noindent
\textbf{Step 2. The quadratic characters of the sums.}
Every $u\in U$ is a positive odd integer at most $K$, and none of its
prime factors belongs to $E$.  By \eqref{eq:residue-hypothesis} and
multiplicativity,
\begin{equation}\label{eq:U-residues}
 \leg{u}{q}
 =\prod_{\ell\mid u\atop \ell\in\mathcal P}\leg{\ell}{q}^{v_\ell(u)}=1.
\end{equation}
For $u,v\in U$, define
\begin{equation*}
 w=\frac{u+v}{2},\qquad 1\leq w\leq K,\qquad w\equiv1\pmod2.
\end{equation*}
Because \(K<q\) and $q\equiv5\pmod8$, we have \(q\nmid w\) and \(q\nmid u+v=2w\).  If $E=\{r\}$,
then $r\nmid w$, since $r\mid w$ would contradict
\eqref{eq:U-no-opposites}.  Thus all prime factors of $w$ also satisfy
\eqref{eq:residue-hypothesis}.  Consequently $\displaystyle \leg{w}{q}=1$.
The supplementary law $$\leg{2}{q}=(-1)^{(q^2-1)/8}=-1,$$ valid
because $q\equiv5\pmod8$, now gives
\begin{equation}\label{eq:U-sum-nonresidue}
 \leg{u+v}{q}=\leg{2}{q}\leg{w}{q}=-1
 \qquad(u,v\in U).
\end{equation}
In particular, the sum is nonzero modulo $q$: $w<q$ and $q$ is odd.

\medskip\noindent
\textbf{Step 3. Taking square roots and multiplicative translates.}
Regard $U\subset[1,K]$ as a subset of $\Fq^\times$; this is an
injective identification since $K<q$.  Define
\begin{equation*}
 S=\{s\in\Fq^\times:s^2\in U\}.
\end{equation*}
By \eqref{eq:U-residues}, each $u\in U$ has exactly two square roots
in $\Fq^\times$.  They are distinct because $q$ is odd, and roots
of different $u$ are disjoint.  Therefore
\begin{equation*}
 |S|=2|U|\geq\frac K6-1.
\end{equation*}
For $s,t\in S$, apply \eqref{eq:U-sum-nonresidue} to $u=s^2$ and
$v=t^2$ to obtain \eqref{eq:Sproperty}.  This includes $s=t$ and
$s=-t$, since the preceding argument allowed $u=v$.  Finally,
for $c\in\Fq^\times$,
\begin{equation}\label{eq:S-translate}
 \leg{(cs)^2+(ct)^2}{q}
 =\leg{c^2}{q}\leg{s^2+t^2}{q}=-1
 \qquad(s,t\in S),
\end{equation}
because $c^2$ is a nonzero square.  This proves the assertion for $cS$.
\end{proof}

\section{Proof of the main theorem}
\label{sec:main-proof}

\begin{proof}[Proof of \Cref{thm:main}]
We divide the proof into six steps and use the notation introduced above.

\medskip\noindent
\textbf{Step 1. The scales and the size of the character family.}
Let $x$ be a sufficiently large real number and put
\begin{equation*}
 L=\log x,\qquad J=\log L,\qquad T=\log J,
 \qquad y=\kappa LT,\qquad\varepsilon=J^{-1/2}.
\end{equation*}
We take $x$ large enough that $y$ exceeds the thresholds in
\Cref{lem:cheb,thm:lau-wu,lem:deletion} and $0<\varepsilon<1/2$.
Use the notation $P_y$, $H=\log P_y$, and $\mathcal C_y$ from
\eqref{eq:PyH}--\eqref{eq:Cy}, and choose $E_y$, $P_y^*$, and
$\mathcal D_y$ by \Cref{lem:deletion}.  Then
\begin{equation}\label{eq:scale-bounds}
 ay\leq H\leq by,\qquad
 \log H=J+\log T+O(1)\sim J.
\end{equation}
Indeed, $H\asymp y$ and $\log y=J+\log T+\log\kappa$.
By \eqref{eq:character-family-size} and the upper bound for $\pi(y)$
in \Cref{lem:cheb},
\begin{equation}\label{eq:family-log-size}
 R_y=|\mathcal C_y|=2^{\pi(y)+1},\qquad
 \log R_y=O\!\left(\frac{LT}{J}\right)=o(\varepsilon L).
\end{equation}
The last relation follows from $T/\sqrt J\to0$.  In particular,
$R_y=\exp(o(L))$, $H=O(LT)$, and $\log H=O(J)=o(L)$, so
\begin{equation}\label{eq:family-small}
 R_y\left(\sqrt{x}\log^2(x+2)+x^{-1/2}H\right)=o(x).
\end{equation}

\medskip\noindent
\textbf{Step 2. Zeros close to the line $\Re s=1$.}
For $0\leq t<1/2$, let $M(t)$ count, with multiplicity and separately
for every $\chi\in\mathcal D_y$, the zeros
$\rho=\beta+i\gamma\in\Zzero_\chi$ satisfying
$\beta\geq1-t$ and $|\gamma|\leq H$.  Put
\begin{equation*}
 \eta=\frac{c_1\sqrt{\log H}}{2H},\qquad
 M(t)=0\quad(0\leq t\leq\eta).
\end{equation*}
The equality is the conclusion of \Cref{lem:deletion}.  Since
$\eta/\varepsilon\asymp J/(LT)\to0$, we also have $0<\eta<\varepsilon$
for large $x$.  The four discriminant types listed in \Cref{sec:quadratic-characters}
partition $\mathcal C_y$ between the two products in
\Cref{thm:lau-wu}(b).  Thus
\begin{equation}\label{eq:M-comparison}
 M(t)\leq N_1(1-t;y)+N_4(1-t;y).
\end{equation}
Here $1-t>1/2$, so the pole at $1$ and the possible zero at $0$
do not affect the zero counts.

We apply the two bounds for $N_\ell(1-t;y)$ in \Cref{thm:lau-wu}(b).
By \eqref{eq:scale-bounds}, $T/\sqrt J\to0$, and
$\log(1/\varepsilon)=T/2$,
\begin{equation}\label{eq:density-coeff}
 \frac{c_2H}{\sqrt{\log H}}\leq\frac L4,
 \qquad
 \frac{c_2H}{\log(1/\varepsilon)}
 \leq2c_2b\kappa L
 =\frac{c_2b}{16(1+c_2b)}L\leq\frac L4
\end{equation}
for all sufficiently large $x$.  If $0<t\leq\varepsilon$, then
$\log(1/t)\geq\log(1/\varepsilon)$. For $t\leq\eta_1(y)$,
the theorem gives $N_\ell(1-t;y)\leq C\sqrt{\log H}\exp(tL/4)$.
For $t>\eta_1(y)$, it gives $N_\ell(1-t;y)\leq C\exp(tL/4)$.
Adding the bounds for $\ell=1,4$ in \eqref{eq:M-comparison} and
enlarging $C$, we obtain
\begin{equation*}
 M(t)\leq C\sqrt{\log H}\,\ee^{tL/4}
 \qquad(0<t\leq\varepsilon).
\end{equation*}

For each zero counted by $M(\varepsilon)$, write $t_\rho=1-\beta$.
Then $0\leq t_\rho\leq\varepsilon$ and
\begin{equation}\label{eq:zero-weight-identity}
 \ee^{-Lt_\rho}=\ee^{-L\varepsilon}
       +L\int_{t_\rho}^{\varepsilon}\ee^{-Lt}\,dt.
\end{equation}
The number of these zeros is finite by \Cref{thm:lau-wu}(c).
Summing \eqref{eq:zero-weight-identity}, with multiplicities, gives
\begin{align}
 &\sum_{\chi\in\mathcal D_y}
  \sum_{\substack{\rho=\beta+i\gamma\in\Zzero_\chi\\
                  \beta\geq1-\varepsilon,\ |\gamma|\leq H}}
 x^\beta \notag\\
 &= x^{1-\varepsilon}M(\varepsilon)
       +xL\int_0^\varepsilon\ee^{-Lt}M(t)\,dt \notag\\
 &\leq Cx\sqrt{\log H}\ee^{-3\varepsilon L/4}
       +CxL\sqrt{\log H}\int_\eta^\varepsilon
                  \ee^{-3Lt/4}\,dt \notag\\
 &\leq Cx\sqrt{\log H}\ee^{-3\varepsilon L/4}
       +\frac{4C}{3}x\sqrt{\log H}\ee^{-3\eta L/4}
       =o(x).
 \label{eq:near-zero}
\end{align}
To justify the last equality, use $\log H\sim J$ and
$\log\log H\sim T$. Since $H\leq b\kappa LT$, we have
\begin{equation*}
 \eta L=\frac{c_1L\sqrt{\log H}}{2H}
 \geq\frac{c_1}{2b\kappa}\frac{\sqrt{\log H}}{T}
 \gg\frac{\sqrt J}{T}.
\end{equation*}
Consequently,
\begin{equation*}
 \log\log H\sim T,\qquad
 \frac{\varepsilon L}{\log\log H}\asymp\frac{L}{\sqrt J\,T}
       \longrightarrow\infty,\qquad
 \frac{\eta L}{\log\log H}
       \gg\frac{\sqrt J}{T^2}\longrightarrow\infty.
\end{equation*}
Thus both exponential factors in \eqref{eq:near-zero} tend to
zero faster than $1/\sqrt{\log H}$.  Since $|W(\rho)|\leq C_h$ in the critical
strip by \eqref{eq:Wdecay}, the same $o(x)$ bound holds with
$x^\beta$ replaced by $|W(\rho)x^\rho|$.

\medskip\noindent
\textbf{Step 3. The remaining zeros and the aggregate estimate.}
For $\beta<1-\varepsilon$ and $|\gamma|\leq H$, we have
$|W(\rho)x^\rho|\leq C_hx^{1-\varepsilon}$.  Each character has
conductor $Q\leq4P_y$, so \Cref{thm:lau-wu}(c), including its assertion
at $0$, gives
\begin{align}
 &\sum_{\chi\in\mathcal D_y}
  \sum_{\substack{\rho=\beta+i\gamma\in\Zzero_\chi\\
                  \beta<1-\varepsilon,\ |\gamma|\leq H}}
       |W(\rho)x^\rho| \notag\\
 &\qquad\ll x^{1-\varepsilon}R_yH\log\bigl(4P_y(H+2)\bigr)
       \ll x^{1-\varepsilon}R_yH^2=o(x).
 \label{eq:low-zeros}
\end{align}
Indeed, $\log(4P_y(H+2))=H+O(\log H)$, and the logarithm of
the last bound divided by $x$ is
\begin{equation*}
 -\varepsilon L+\log R_y+2\log H
       =-\varepsilon L+o(\varepsilon L)\longrightarrow-\infty.
\end{equation*}
Here \eqref{eq:family-log-size} applies and
$\log H=O(J)=o(L/\sqrt J)$.

For the zeros with $|\gamma|>H$, let $m_0=\lfloor H\rfloor$ and
partition the absolute ordinates into intervals $[m,m+1)$,
$m\geq m_0$.  The number of zeros of any one character in such
an interval is at most its total number with $|\gamma|\leq m+1$,
which is $O((m+1)\log(Q(m+3)))$ by \Cref{thm:lau-wu}(c).
Also $x^\beta\leq x$, $|W(\rho)|\leq C_h\ee^{-m}$, and
$\log Q\leq H+\log4$.  Hence
\begin{align}
 &\sum_{\chi\in\mathcal D_y}
  \sum_{\substack{\rho\in\Zzero_\chi\\|\Im\rho|>H}}
       |W(\rho)x^\rho| \notag\\
 &\qquad\ll xR_y\sum_{m=m_0}^{\infty}
       \ee^{-m}(m+1)\bigl(H+\log(m+3)+1\bigr) \notag\\
 &\qquad\ll xR_y\ee^{-H}(H+1)^2=o(x).
 \label{eq:high-zeros}
\end{align}
For the second inequality, set $m=m_0+j$ and bound the summand by
$O(\ee^{-H-j}(H+1)^2(j+1)^2)$; the sum over $j\geq0$ converges.
For the last equality, \eqref{eq:scale-bounds} and
\eqref{eq:family-log-size} give
$$\log R_y+2\log(H+1)=o(H).$$

Apply \Cref{lem:explicit} to every $\chi\in\mathcal D_y$.
Its error terms sum to $o(x)$ by \eqref{eq:family-small}, since
$|\mathcal D_y|\leq R_y$ and $\log(Q+2)=O(H)$.
The three ranges in \eqref{eq:near-zero}, \eqref{eq:low-zeros}, and
\eqref{eq:high-zeros} exhaust $\Zzero_\chi$.  The triangle inequality
therefore yields
\begin{equation}\label{eq:aggregate}
 \sum_{\chi\in\mathcal D_y}
 \left|T_x(\chi)-\1_{\chi=\chi_1}xW(1)\right|=o(x).
\end{equation}
The bounds above use $|\mathcal D_y|\leq R_y$ and
$M(t)\leq N_1(1-t;y)+N_4(1-t;y)$. Their right-hand sides do not
depend on which prime is omitted. Hence the $o(x)$ error in
\eqref{eq:aggregate} is uniform over the choices of $E_y$ in
\Cref{lem:deletion}.

\medskip\noindent
\textbf{Step 4. A prime with the required residue pattern.}
For every integer $n$, define
\begin{equation}\label{eq:F-definition}
 F_y(n)=(1+\chi_{-4}(n))(1-\chi_8(n))
       \prod_{\substack{r\in\mathcal P,\ r\leq y\\r>2,\ r\notin E_y}}
             (1+\chi_{r^*}(n)).
\end{equation}
Each character value belongs to $\{-1,0,1\}$, so $F_y(n)\geq0$.
The integer $P_y^*/2$ is the product of the odd primes
$r\leq y$ with $r\notin E_y$, each taken once. Thus the products
formed by choosing any of these primes are precisely the divisors
$d\mid P_y^*/2$; the empty choice gives $d=1$.
Expanding \eqref{eq:F-definition} and using
\eqref{eq:quadratic-products} gives
\begin{equation*}
 F_y(n)=\sum_{\substack{d\in\N\\d\mid P_y^*/2}}
 \bigl(\chi_{d^*}(n)+\chi_{-4d^*}(n)
       -\chi_{8d^*}(n)-\chi_{-8d^*}(n)\bigr).
\end{equation*}
The discriminant classification in \Cref{sec:quadratic-characters} shows that these are
exactly the characters of $\mathcal D_y$, each occurring once.
Thus $$F_y=\sum_{\chi\in\mathcal D_y}c_\chi\chi,$$ where
$c_\chi\in\{-1,1\}$ and $c_{\chi_1}=1$.
For fixed $x$, the formula for $h$ gives
\begin{equation*}
 \sum_{p\in\mathcal P}(\log p)|h(p/x)|
 \ll_x\sum_{n=2}^{\infty}n^{37}(\log n)\ee^{-24n/x}<\infty.
\end{equation*}
Since there are finitely many characters and $|\chi(p)|\leq1$,
we may sum the character expansion over primes. By \eqref{eq:aggregate},
\begin{align}
 \sum_{p\in\mathcal P}(\log p)h(p/x)F_y(p)
 =\sum_{\chi\in\mathcal D_y}c_\chi T_x(\chi)
 =xW(1)+o(x)>0,
 \label{eq:positive}
\end{align}
since $W(1)>0$ by \eqref{eq:Wdecay}.

For $p\notin(x,2x)$, the corresponding summand in
\eqref{eq:positive} is nonpositive, because $h(p/x)\leq0$ and
$F_y(p)\geq0$.  There must therefore be a prime $q\in(x,2x)$ with
$F_y(q)>0$.  Since $y=\kappa\log x\log_3x=o(x)$, we have $q>y$
for large $x$.  The prime $q$ consequently divides none of the
conductors in \eqref{eq:F-definition}; all character values there
are $\pm1$.  Positivity of the product forces
\begin{equation*}
 \chi_{-4}(q)=1,\qquad\chi_8(q)=-1,\qquad
 \chi_{r^*}(q)=1\quad(r\leq y,\ r>2,\ r\notin E_y).
\end{equation*}
The first identity gives $q\equiv1\pmod4$.  The supplementary law
at $2$ gives $$\chi_8(q)=\leg{2}{q}=-1$$ only for
$q\equiv3,5\pmod8$; hence $q\equiv5\pmod8$.
Also $r^*=\pm r$, and $\displaystyle \leg{-1}{q}=1$ because $q\equiv1\pmod4$.
We conclude that
\begin{equation}\label{eq:qpattern}
 \begin{gathered}
 x<q<2x,\qquad q\equiv5\pmod8,\\
 \leg{r}{q}=1\quad(r\leq y,\ r\text{ an odd prime},\ r\notin E_y).
 \end{gathered}
\end{equation}

\medskip\noindent
\textbf{Step 5. An admissible subset of the prime-power box.}
Now let $N$ be sufficiently large, take $x=p_N$, and set
$K=\lfloor y\rfloor$.  Then $1\leq K<q$.  Every prime in $E_y$
is at most $y$ and is an integer, so it is at most $K$.
Thus \eqref{eq:qpattern} verifies every hypothesis of
\Cref{lem:residue}, with $E=E_y$.  It supplies $S\subset\Fq^\times$
with $|S|\geq K/6-1$ and \eqref{eq:Sproperty}.
Every $n\in\Phi_N$ is nonzero modulo $q$, since its prime factors
are at most $p_N<q$.  For each $c\in\Fq^\times$, put
\begin{equation*}
 B_c=\{n\in\Phi_N:n\bmod q\in cS\}.
\end{equation*}
If $u,v\in B_c$ and $w\in\N$ satisfied $u^2+v^2=w^2$, choose
$s,t\in S$ with $u\equiv cs$ and $v\equiv ct\pmod q$.
Then
\begin{equation*}
 w^2\equiv c^2(s^2+t^2)\pmod q,\qquad
 \leg{c^2(s^2+t^2)}{q}=-1.
\end{equation*}
The second equality follows from \eqref{eq:S-translate}.
It contradicts the first, since a square has Legendre symbol $0$ or
$1$.  Thus every $B_c$ is admissible.

For any fixed $n\in\Phi_N$ and $s\in S$, there is a unique
$c\in\Fq^\times$ with $n\equiv cs\pmod q$, namely
$c=(n\bmod q)s^{-1}$.  Distinct $s$ give distinct $c$, since
$n\bmod q\ne0$.  Therefore exactly $|S|$ translates contain
$n\bmod q$, and double counting yields
\begin{equation}\label{eq:translate-average}
 \sum_{c\in\Fq^\times}|B_c|=|\Phi_N||S|=N^N|S|,
 \qquad f(N)\geq\max_c|B_c|\geq\frac{N^N|S|}{q-1}.
\end{equation}

\medskip\noindent
\textbf{Step 6. The quantitative lower bound.}
Increase $N$ so that $K\geq6$.  Since $q<2p_N$,
$K\geq y-1$, and $N=\pi(p_N)\geq ap_N/\log p_N$ by
\Cref{lem:cheb}, \eqref{eq:translate-average} gives
\begin{align}
 \frac{f(N)}{N^{N-1}}
 &\geq\frac{N}{2p_N}\left(\frac K6-1\right)
  \geq\frac{a}{2\log p_N}\left(\frac K6-1\right) \notag\\
 &\geq\frac{a}{2\log p_N}
       \left(\frac{\kappa\log p_N\log_3p_N}{6}-\frac76\right)
  =\frac{a\kappa}{12}\log_3p_N-\frac{7a}{12\log p_N}.
 \label{eq:final-intermediate}
\end{align}
The use of the lower bound for $N$ is valid because $K/6-1\geq0$.
Since $p_N\geq N$, monotonicity of the iterated logarithm gives
$\log_3p_N\geq\log_3N$ for sufficiently large $N$.
The error term in \eqref{eq:final-intermediate} tends to zero,
whereas $\log_3N\to\infty$.  We can therefore choose an absolute
$N_0$ satisfying all earlier thresholds and
\begin{equation*}
 \frac{7a}{12\log p_N}\leq\frac{a\kappa}{24}\log_3N
 \qquad(N\geq N_0).
\end{equation*}
Substitution into \eqref{eq:final-intermediate} proves
\eqref{eq:main} for every $N\geq N_0$.
\end{proof}

\begin{remark}[On replacing $\log_3N$ by $\log_2N$]
\label{rem:log-two-obstruction}
The factor $\log_3N$ comes from the choice
$y=\kappa\log x\log_3x$. Indeed, with $x=p_N$, the averaging
argument and the bound $|S|\gg y$ give
\begin{equation*}
 f(N)\gg N^{N-1}\frac{y}{\log x},
\end{equation*}
provided that a prime $q\in(x,2x)$ with the required residue
pattern is available. Chebyshev's estimates give
$\log_2p_N\sim\log_2N$. To obtain a factor of order $\log_2N$
by this construction, one would therefore seek that pattern at
the larger scale
\begin{equation*}
 y=\delta LJ,\qquad
 L=\log x,\quad J=\log L,\quad \delta>0\text{ fixed}.
\end{equation*}
The estimates used above do not establish it, for the following reasons.

At this scale, the bounds for $\vartheta(y)$ give
$H\asymp\delta LJ$ and $\log H\sim J$. Consequently,
\begin{equation*}
 \eta L=\frac{c_1L\sqrt{\log H}}{2H}
       \asymp\frac{1}{\delta\sqrt J}\longrightarrow0,
 \qquad
 \frac{c_2H}{L\sqrt{\log H}}
       \asymp\delta\sqrt J\longrightarrow\infty.
\end{equation*}
The first relation shows that the zero-free region no longer
supplies the decay used after
\eqref{eq:near-zero}: the factor
$\sqrt{\log H}\exp(-3\eta L/4)$ is now asymptotic to $\sqrt J$.
The second relation prevents the first inequality in
\eqref{eq:density-coeff}.
Thus the argument for making the contribution of zeros near $1$
equal to $o(x)$ fails at this scale.

Changing only the cutoff $\varepsilon\in(0,1/2)$ does not repair
the estimates in Steps~2 and~3. By \eqref{eq:character-family-size}
and the bounds for $\pi(y)$, we have $\log R_y\asymp\delta L$.
To obtain an $o(x)$ estimate from the final bound in
\eqref{eq:low-zeros}, one needs
\begin{equation*}
 \log\bigl(x^{-\varepsilon}R_yH^2\bigr)
 =-\varepsilon L+\log R_y+2\log H
 \longrightarrow-\infty.
\end{equation*}
Since $\log R_y\geq c\delta L$ for some absolute $c>0$,
this expression can tend to $-\infty$ only if
$\varepsilon\geq c\delta/2$ for all sufficiently large $x$.
Such a cutoff exceeds $\eta_1(y)$, which tends to zero.
We then have $\log(1/\varepsilon)=O_\delta(1)$, whereas
\begin{equation*}
 \frac{c_2H}{L\log(1/\varepsilon)}
       \gg_\delta J\longrightarrow\infty.
\end{equation*}
Hence the second inequality in \eqref{eq:density-coeff} cannot
hold either.

A direct enlargement of $y$ therefore does not yield the desired
improvement. To make the argument work at this larger scale, one would
need stronger control of zeros near $\Re s=1$, or an argument that avoids
the present absolute estimates
over the character family. These observations concern the estimates
used here; they do not assert that the factor $\log_3N$ is optimal.
\end{remark}

\FinalContact


\begin{thebibliography}{99}
\small
\bibitem{Apostol}
\textsc{T.~M.~Apostol}, \emph{Introduction to Analytic Number Theory},
Undergraduate Texts in Mathematics, Springer-Verlag, New York, 1976.
Available at \url{https://doi.org/10.1007/978-1-4757-5579-4}.

\bibitem{AzevedoMoreira}
\textsc{G.~Azevedo and J.~Moreira}, Pythagorean triples in level sets of
completely multiplicative functions, preprint, arXiv:2607.04903
[math.NT], 2026.
Available at \url{https://arxiv.org/abs/2607.04903}.

\bibitem{CooperOverstreet}
\textsc{J.~Cooper and R.~Overstreet}, Coloring so that no Pythagorean
triple is monochromatic, preprint, arXiv:1505.02222 [math.CO], 2015.
Available at \url{https://arxiv.org/abs/1505.02222}.

\bibitem{FrantzikinakisHost}
\textsc{N.~Frantzikinakis and B.~Host}, Higher order Fourier analysis of
multiplicative functions and applications, \emph{J. Amer. Math. Soc.}
\textbf{30} (2017), no.~1, 67--157.
Available at \url{https://doi.org/10.1090/jams/857}.

\bibitem{FKMgeneralized}
\textsc{N.~Frantzikinakis, O.~Klurman, and J.~Moreira}, Partition regularity
of generalized Pythagorean pairs, preprint, arXiv:2407.08360
[math.CO, math.NT], 2024, revised 2026.
Available at \url{https://arxiv.org/abs/2407.08360}.

\bibitem{FKM}
\textsc{N.~Frantzikinakis, O.~Klurman, and J.~Moreira}, Partition regularity
of Pythagorean pairs, \emph{Forum Math. Pi} \textbf{13} (2025), e5.
Available at \url{https://doi.org/10.1017/fmp.2024.27}.

\bibitem{FrantzikinakisMountakis}
\textsc{N.~Frantzikinakis and A.~Mountakis}, Recurrence for pretentious
systems along generalized Pythagorean triples, preprint,
arXiv:2508.09778 [math.DS, math.CO, math.NT], 2025, revised 2026.
Available at \url{https://arxiv.org/abs/2508.09778}.

\bibitem{GrahamFavorite}
\textsc{R.~Graham}, Some of my favorite problems in Ramsey theory,
in \emph{Combinatorial Number Theory}, de Gruyter, 2007, pp.~229--236.
Available at
\url{https://mathweb.ucsd.edu/~ronspubs/07_04_favorite_problems.pdf}.

\bibitem{HeuleKullmannMarek}
\textsc{M.~J.~H.~Heule, O.~Kullmann, and V.~W.~Marek}, Solving and
verifying the Boolean Pythagorean triples problem via cube-and-conquer,
in \emph{Theory and Applications of Satisfiability Testing---SAT 2016},
Lecture Notes in Computer Science, vol.~9710, Springer, 2016,
pp.~228--245.
Available at \url{https://doi.org/10.1007/978-3-319-40970-2_15}.

\bibitem{IwaniecKowalski}
\textsc{H.~Iwaniec and E.~Kowalski}, \emph{Analytic Number Theory},
American Mathematical Society Colloquium Publications, vol.~53,
American Mathematical Society, Providence, RI, 2004.
Available at \url{https://doi.org/10.1090/coll/053}.

\bibitem{LauWu}
\textsc{Y.-K.~Lau and J.~Wu}, On the least quadratic non-residue,
\emph{Int. J. Number Theory} \textbf{4} (2008), no.~3, 423--435.
Available at \url{https://doi.org/10.1142/S1793042108001432}.

\bibitem{Leng2026}
\textsc{J.~Leng}, A new proof of two-color partition regularity of
Pythagorean triples, preprint, arXiv:2609.34676
[math.NT, math.CO, math.DS], 2026.
Available at \url{https://arxiv.org/abs/2609.34676}.

\bibitem{LengMcNamaraMountakis}
\textsc{J.~Leng, R.~McNamara, and A.~Mountakis}, Monochromatic
almost-Pythagorean quadruples, preprint, arXiv:2610.04078
[math.DS, math.CO], 2026.
Available at \url{https://arxiv.org/abs/2610.04078}.

\bibitem{Olver}
\textsc{F.~W.~J.~Olver}, \emph{Asymptotics and Special Functions},
A.~K.~Peters, Wellesley, MA, 1997; corrected reprint of the 1974
edition published by Academic Press.
Available at \url{https://doi.org/10.1201/9781439864548}.

\bibitem{Titchmarsh}
\textsc{E.~C.~Titchmarsh}, \emph{The Theory of the Riemann Zeta-function},
2nd ed., revised by D.~R.~Heath-Brown, Oxford University Press,
Oxford, 1986.
Available at
\url{https://sites.math.rutgers.edu/~zeilberg/EM18/TitchmarshZeta.pdf}.

\bibitem{Wesley2026}
\textsc{W.~J.~Wesley}, An elementary proof of the 2-regularity of
Pythagorean triples, preprint, arXiv:2609.39063 [math.CO], 2026.
Available at \url{https://arxiv.org/abs/2609.39063}.

\bibitem{XiaoPrimeBoxes}
\textsc{R.~Xiao}, Pythagorean pairs in boxes of prime powers, preprint,
arXiv:2610.11965v1 [math.NT, math.CO], 2026.
Available at \url{https://arxiv.org/abs/2610.11965v1}.
\end{thebibliography}
\end{document}